\documentclass{amsart}
\usepackage[T1]{fontenc} 
\usepackage{graphicx}

\usepackage{amsmath,amssymb,amsthm}  
\usepackage[all]{xy}
\usepackage{float}
\usepackage{multicol}
\usepackage{mathtools} 

\newcommand{\Real}{\mathbb{R}} \newcommand{\Complex}{\mathbb{C}}
\newcommand{\Integer}{\mathbb{Z}} 
\newcommand{\DT}[1]{#1 \dots #1}
\newcommand{\bydef}{\stackrel{\mbox{\tiny def}}{=}} \def \lmod#1\rmod
{\left|\smash{#1}\right|\vphantom{#1}}

\newcommand{\name}[1]{\operatorname{\mathrm{#1}}}
\newcommand{\Const}{\mathcal{C}} \let \eps=\varepsilon

\def \LL {\mathop{\mathcal{LL}}\nolimits}
\def \OLL {\mathop{\mathcal{OL}}\nolimits}

\newcommand{\wf}{\widehat{f}}
\newcommand{\T}{\mathcal T}

\newcommand{\Upp}{\overline{\mathbb H}}
\newcommand{\Hurw}{\mathcal H^{\text{num}}}
\newcommand{\RHurw}{\mathbb R\mathcal H^{\text{num}}}
\newcommand{\HStand}{\mathcal H}
\newcommand{\RHStand}{\mathbb R\mathcal H}

\theoremstyle{plain}
\newtheorem{theorem}{Theorem}
\newtheorem{lemma}[theorem]{Lemma}
\newtheorem{proposition}[theorem]{Proposition}

\newtheorem{corollary}[theorem]{Corollary}
\theoremstyle{definition}
\def \definitionName {Definition}
\newtheorem{definition}[theorem]{\definitionName}

\theoremstyle{remark}
\def \remarkName {Remark}

\newtheorem{Remark}{\remarkName}

\numberwithin{theorem}{section}
\numberwithin{equation}{section}

\author{Yurii Burman}
\author{Rapha\"el Fesler}
\address[1]{Independent University of Moscow}
\address[1,2]{Higher School of Economics, Moscow}
\address[1]{Ben Gurion University of the Negev, Beer Sheva, Israel}
\address[2]{Leonard Euler International Mathematical Institute in St.Petersburg, Russia}
\email[1]{yburman@gmail.com}
\email[2]{raphael.fesler@gmail.com}


\title{Real algebraic curves and twisted Hurwitz numbers}

\begin{document}
	\maketitle
	
	\pagestyle{empty}
	
	\tableofcontents
	
	\begin{abstract}
		We provide a direct correspondence between the $b$-Hurwitz numbers
		with $b=1$ from \cite{ChapuyDolega}, and twisted Hurwtiz numbers
		from \cite{TwistedHurwitz}. This provides a description of real
		coverings of the sphere with ramification on the real line in terms
		of monodromy.
	\end{abstract}
	
	\section*{Introduction}
	Hurwitz numbers go back to the XIX century work by Adolf Hurwitz
	\cite{H91}, and are now a classical topic in combinatorics and
	algebraic geometry. Their combinatorial definition is as follows: let
	$m$ be an integer and $\lambda$, a partition of $n$. Then the Hurwitz
	number $h_{m,\lambda}$, is $1/n!$ times the number of sequences of $m$
	transpositions $(\sigma_1,\dots,\sigma_m)$ in the permutation group
	$S_n$ such that lengths of the independent cycles of their product
	$\sigma_1\dots\sigma_m \in S_n$ form the partition $\lambda$. Hurwitz
	showed that $h_{m,\lambda}$ is equal to the number of isomorphism
	classes of meromorphic functions such that $m$ of their critical
	values are simple and one critical value has $\lambda$ as a
	ramification profile; this is the geometric definition of
	$h_{m,\lambda}$.
	
	During the past decades, several variants and generalizations of
	Hurwitz numbers appeared in the literature: monotonic Hurwitz numbers
	\cite{GGPN13}, Hurwitz numbers for reflection groups \cite{B-Hurwitz},
	projective Hurwitz numbers \cite{NO17}, and Hurwitz--Severi numbers
	\cite{BS19}, to name just a few. A particular generalization appeared
	in the work by Chapuy and Do\l\k ega \cite{ChapuyDolega} and is called
	$b$-Hurwitz numbers. They are defined as coefficients of a certain
	polynomial $h_{m,\lambda}(b)$, counting ramified covers of the disk
	liftable to a meromorphic function on a complex curve and such that
	all the critical values lie on the boundary of the disk and all of
	them are simple, except possibly one; see \cite{ChapuyDolega} for
	details. The value $h_{m,\lambda}(0)$ is equal to the classical
	Hurwitz numbers, whereas $h_{m,\lambda}(1)$ is equal to the real
	Hurwitz number defined later in this article; see Sections
	\ref{Sec:Prelim} and \ref{Sec:AlgGeom} for details.
	
	A couple of years later, in a different context, the authors of this
	article re-discovered the $b=1$ case of the $b$-Hurwitz numbers
	\cite{TwistedHurwitz}; these were called twisted Hurwitz numbers due
	to their relation to surgery theory in dimension $2$. The paper
	\cite{TwistedHurwitz}, though, contained no direct correspondence
	between the objects enumerated by the twisted Hurwitz numbers and
	those enumerated by the $b=1$ Hurwitz numbers of \cite{ChapuyDolega};
	it was proved instead that relevant generating functions coincide. We
	give the direct correspondence in this article.
	
	The article has the following structure: in Section \ref{Sec:Prelim}
	we recall previous results and introduce the necessary notations.
	Section \ref{Sec:Combinat} is devoted to a combinatorial
	correspondence using perfect matchings; we make use of the results of
	Ben Dali \cite{BenDali} here and augment it by a couple of new
	constructions. In Section \ref{Sec:AlgGeom} we use the machinery of
	the Lyashko--Looijenga map (a.k.a.\ branch map) applied to the real
	Hurwitz space in order to provide a geometric correspondence between
	the same objects. Finally, in Section \ref{Sec:GeomCombin} we show
	that the combinatorial correspondence and the geometric correspondence
	are inverse to one another. This gives in particular a description of
	fully real ramified coverings in terms of monodromy.
	
	\subsection*{Acknowledgements}
	The authors would like to thank Houcine Ben Dali who explained to them
	crucial details from his article \cite{BenDali}. The research of the
	second author was supported by the grant from the Government of the
	Russian Federation, Agreement No.075-15-2019-1620.
	
	\section{Preliminaries}\label{Sec:Prelim}
	In this section we introduce necessary definitions and notations and
	recall results from \cite{ChapuyDolega}, \cite{BenDali} and
	\cite{TwistedHurwitz}. See the cited papers for details.
	
	As usual, we denote by $S_{2n}$ the group of permutations of a
	$2n$-element set; take $\mathcal A_n = \{1, \bar 1 \DT, n, \bar n\}$
	for such. Fix a partition $\lambda=(\lambda_1 \DT, \lambda_s)$, $\lmod
	\lambda\rmod \bydef \lambda_1 \DT+ \lambda_s = n$, and let
	$\tau=(1,\bar 1)\dots(n, \bar n) \in S_{2n}$. Denote by $B_\lambda^{\sim}
	\subset S_{2n}$ the set of permutations whose decomposition into
	independent cycles consists of $2s$ cycles $c_1, c_1' \DT, c_s, c_s'$
	such that the lengths of $c_i$ and of $c_i'$ are $\lambda_i$, for all
	$i = 1 \DT, s$, and $c_i' = \tau c_i^{-1} \tau$. All the elements of
	$\sigma \in B_\lambda^{\sim}$ satisfy the relation $\tau \sigma =
	\sigma^{-1} \tau$.
	
	\begin{proposition}[\cite{TwistedHurwitz}]\label{Pp:ConjCl}
		$B_\lambda^{\sim}$ is a $B_n$-conjugacy class in $S_{2n}$.
	\end{proposition}
	
	\begin{definition}\label{Def:RealHurw}
		The {\em purely real Hurwitz numbers} are
		\begin{align*}
			h_{m,\lambda}^\Real &\bydef \frac{1}{\lmod\lambda\rmod!} \#\mathfrak
			H_{m,\lambda}^{\Real}\\
			&\text{where}\\
			\mathfrak H_{m,\lambda}^{\Real} &\bydef \{(\sigma_1 \DT,\sigma_m)
			\mid \forall s = 1 \DT, m \ \sigma_s = (i_s j_s), j_s \ne
			\tau(i_s),\\
			&\sigma_1\sigma_2 \dots \sigma_m (\tau\sigma_m\tau) \dots
			(\tau\sigma_1\tau) \in B_\lambda^{\sim}\}.
		\end{align*}
	\end{definition}
	
	Denote by $\HStand_{m,\lambda}$ the main stratum of the standard
	Hurwitz space; its elements are equivalence classes of pairs $(M,f)$
	where $M$ is a compact smooth complex curve and $f: M \to \Complex
	P^1$, a meromorphic function having $s$ poles $u_1 \DT, u_s$ of
	multiplicities $\lambda_1 \DT, \lambda_s$ and $m$ simple critical
	points. See e.g.\ \cite{ELSV} for exact definition of the topology and
	algebraic variety structure on $\HStand_{m,\lambda}$. Also denote by
	$\Hurw_{m,\lambda}$ a {\em decorated Hurwitz space}; its elements are
	triples $(M,f,\nu)$ where $(M,f) \in \HStand_{m,\lambda}$ and $\nu$ is
	a numbering of the critical points, that is, a bijection from the set
	of critical points of $f$ to $\{1 \DT, m\}$. The forgetful map
	$\Phi_1: (M,f,\nu) \mapsto (M,f)$ sends $\Hurw_{m,\lambda}$ to
	$\HStand_{m,\lambda}$; we choose the topology and the algebraic
	variety structure in $\Hurw_{m,\lambda}$ so that $\Phi_1$ would become
	a continuous algebraic map.
	
	Denote $\Upp \bydef \Complex P^1/(z \sim \bar z) = \mathbb H \cup
	\{\infty\}$ where $\mathbb H \subset \Complex$ is the upper
	half-plane; $\Upp$ is homeomorphic to a disk. Also denote by $\pi:
	\Complex P^1 \to \Upp$ the quotient map.
	
	Let $\hat N$ be a real curve without real points, that is, a complex
	curve with a fixed-points-free anti-holomorphic involution $\T: \hat
	N \to \hat N$. The quotient $N \bydef \hat N/\T$ is then a smooth real
	surface, not necessarily orientable; denote by $p: \hat N \to N$ the
	quotient map. A meromorphic function $h: \hat N \to \Complex P^1$ is
	called real if $h(\T(a)) = \overline{h(a)}$ for all $a \in \hat N$;
	see \cite{Natanzon} for details. A real meromorphic function $\wf$ is
	uniquely included into a commutative diagram
	\begin{equation}\label{Eq:Branched}
		\begin{array}{ccc}
			\hat N & \stackrel{\wf}{\longrightarrow} & \Complex P^1\\
			\downarrow p && \downarrow \pi\\
			N & \stackrel{f}{\longrightarrow} & \Upp
		\end{array}
	\end{equation}
	called in \cite{ChapuyDolega} a {\em twisted ramified covering}. 
	
	Following \cite{ChapuyDolega}, call a real meromorphic function {\em
		simple} if all its critical points $u_i$, except possibly poles, are
	simple, and the critical values are as simple as possible: $f(u_i) \ne
	f(u_j)$ unless $u_j = u_i$ or $u_j = \T(u_i)$ and $f(u_i) = f(\T(u_i))
	\in \Real$. A simple real meromorphic function is called {\em fully
		real} if all its critical values are real.
	
	Note that the simplicity condition is generally not assumed for
	poles. The involution $\T$ has no fixed points, so the ramification
	profile of $\wf$ over $\infty$ has every part repeated twice:
	$(\lambda_1, \lambda_1 \DT, \lambda_s, \lambda_s)$, and $\deg \wf =
	2n$ is even. We say then that the profile of the simple twisted real
	function above is $\lambda = (\lambda_1 \DT, \lambda_s)$, and also
	write $\deg f = n$ by a slight abuse of notation. 
	
	Denote by $\RHStand_{m,\lambda}$ the set of simple real meromorphic
	functions of the profile $\lambda$ with $2m$ simple critical points
	$u_1, \T(u_1) \DT, u_m, \T(u_m)$, up to equivalence (similar to the
	classical Hurwitz space, see \cite{Natanzon} for details). The subset
	of fully real functions is denoted by ${\mathfrak{RH}}_{m,\lambda}
	\subset \RHStand_{m,\lambda}$. For $F \in {\mathfrak{RH}}_{m,\lambda}$
	we usually denote by $y_0 \DT< y_m \in \Real$ its critical values
	(each one assumed in two simple critical points).

	For $\wf \in \mathfrak{RH}_{m,\lambda}$ let $u \in \Real \subset
	\Complex P^1$ be a regular (not critical) value of $\wf$ such that $u
	< y_0$; then the preimage $\wf^{-1}(u) \subset \hat{N}$ consists of
	$2n$ points and the preimage $f^{-1}(\pi(u)) \subset N$, of $n$
	points. Fix a bijection $\hat\nu: \wf^{-1}(\hat u) \to \mathcal{A}_n$
	such that if $\hat\nu(x) = k$ then $\hat\nu(\T(x)) = \bar k$ for all
	$k = 1 \DT, n$. A simple fully real ramified covering together with
	the point $u$ and the bijection $\hat \nu$ is called labelled. The set
	of labelled fully real simple ramified coverings $(\wf,\nu)$ where
	$\wf \in \mathfrak{RH}_{m,\lambda}$ is denoted $\mathfrak{D}_{m,\lambda}^o$
	(``o'' from ``oriented'').

	Let $m \geq 1$; consider a graph $\Const$ embedded into the surface
	$N$ (in the sense of \cite{LandoZvonkin}: the complement $N \setminus
	\Const$ is a union of open disks) and such that its vertices are
	colored $0 \DT, m-1$. Note that neither $N$ nor $\Const$ are assumed
	connected, and $N$ is generally not orientable. Edges of an embedded
	graph incident to any vertex cut its neighbourhood into corners. We
	speak about ``corner $i$'' if the vertex bears the color $i$. 
	
	\begin{definition}[\cite{ChapuyDolega}]\label{Df:Const}
		The graph $\Const$ is called a {\em simple $m$-constellation} if
		\begin{enumerate}
			\item One vertex of color $0$ is joined by two edges with vertices
			(distinct or not) of color $1$. The remaining vertices of color
			$0$ are incident to one edge only joining it with a vertex of
			color $1$. The same picture is for the color $m-1$ (and $m-2$
			instead of $1$).
			
			\item For any color $k = 1 \DT, m-2$ there exists exactly one vertex
			of this color incident to two edges joining it with vertices of
			color $k-1$ and two, with vertices of color $k+1$. These edges
			alternate as one moves around the vertex. The remaining vertices
			of color $k$ are incident to two edges each joining them with one
			vertex of color $k-1$ and another, of color $k+1$.
		\end{enumerate}
	\end{definition}
	Thus, every corner of color $k = 1 \DT, m-2$ contains three successive
	vertices of colors $k-1$, $k$ (the pivot) and $k+1$. The degree of a
	face is defined as the numbers of corners $0$ it contains. Denote by
	$\lambda$ the partition $\lambda(\Const) \bydef (\lambda_1 \DT\le
	\lambda_s)$ where $s$ is the number of faces of $\Const$ and
	$\lambda_i$ are their degrees. The number $n \bydef
	\lmod\lambda(\Const)\rmod \bydef \lambda_1 \DT+ \lambda_s$ is called
	the size of $\Const$.

	Simple $m$-constellations are split into equivalence classes by
	self-homeomorphisms of the surface $N$.
	
	The next theorem is a particular case of an important result from
	\cite{ChapuyDolega} relating fully real simple ramified coverings and
	simple $m$-constellations.
	
	\begin{theorem}[\cite{ChapuyDolega}]\label{th:CDConstel}
		Let $\lambda$ be a partition with $\lmod\lambda\rmod = n$. Let
		\eqref{Eq:Branched} be a simple fully real ramified covering with
		$m$ critical values $y_0 \DT< y_{m-1} \in \Real \subset \Complex
		P^1$ and a profile $\lambda$. If $P \subset \Upp$ is a segment with
		the endpoints $\pi(y_0)$ and $\pi(y_{m-1})$ not passing through
		$\infty$ then the preimage $f^{-1}(P) \subset N$ is a simple
		$m$-constellation of size $n$ having the profile $\lambda$.
		
		The map $f \mapsto f^{-1}(P)$ is a one-to-one correspondence between
		equivalence classes of simple fully real ramified coverings with $m$
		critical values and the profile $\lambda$ and equivalence classes of
		simple $m$-constellations with the profile $\lambda$.
	\end{theorem}
	
	\begin{definition}
		A constellation of size $n$ is \textit{CD labeled} if its corners of
		color $0$ are bijectively labeled with integers from $1$ to $n$, and
		for each corner of color $0$ an orientation if chosen.
	\end{definition}
	
	Let $\wf$ be a simple fully real ramified covering as above, and
	$\Const = f^{-1}(P) \subset N$, the associated constellation. To any
	CD labelling of $\Const$ one can relate a labelling $\hat \nu$ turning
	$\wf$ into an element of $\mathfrak{D}_{m,\lambda}^o$. Namely, if $u <
	y_0$ is a non-critical value as above then $\pi(u) \in \partial\Upp$
	and any corner of $\Const$ (there are totally $n$ of them) contains
	one element $v$ of the preimage $f^{-1}(\pi(u))$. Then $p^{-1}(v)
	\subset \wf^{-1}(u)$ consists of two points, $w_1$ and $w_2 =
	\T(w_1)$. If the corner containing $v$ has number $k$, then $w_1$ and
	$w_2$ should be labelled $k$ and $\bar k$. The mapping $p: \hat N \to
	N$ near points $w_1$ and $w_2$ sends the global orientation of $\hat
	N$ (obtained from its complex structure) to local orientations of $N$
	near the point $v = p(w_1) = p(w_2)$. Since $\T$ is
	orientation-reversing and $p \circ \T = p$, these two orientations are
	opposite, so exactly one of them coincides with the orientation of the
	corner of color $0$ containing $v$ (recall that the orientation of the
	corner is a part of the CD labelling). If the coincidence happened for
	the mapping $p$ near the point $w_i$ (where $i = 1$ or $2$) then take
	$\hat \nu(w_i) = k$ and $\hat \nu(w_j) = \bar k$ for the second point.

	The following construction from \cite{BenDali} provides a
	combinatorial description of the simple $m$-constellations. Take the
	segment $P \subset \Upp$ and move it slightly up to obtain $P'$. The
	mapping $f$ has no critical values in the interior of $\Upp$,
	therefore the preimage $f^{-1}(P')$ consists of $2n$ connected
	components homeomorphic to a segment. These components are called {\em
		right paths} in \cite{BenDali}; each of them is a path going near
	the edges of $\Const$ but not intersecting them and joining a corner
	of color $0$ with a corner of color $m-1$.
	
	Each corner $0$ of $\Const$ is a starting point of two right paths. A
	CD labelling of $\Const$ allows to number the right paths by elements
	of the set $\mathcal A_n$ as follows: the paths starting at a corner
	number $k$ are given numbers $k$ and $\bar k$, of them $k$ is the
	first path according to the orientation of the corner; see
	Fig.~\ref{Fg:RightP}.
	
	\begin{figure}[H]
		\center
		\includegraphics[scale=0.5]{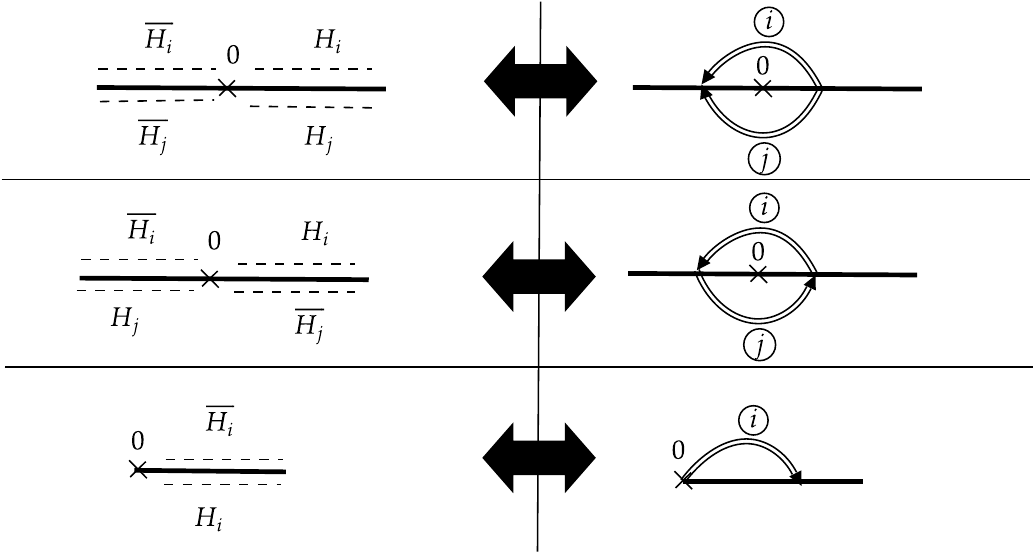} 
		\caption{Label Correspondence}\label{Fg:RightP}
	\end{figure}
	
	Take a point $u \in \Real \subset \Complex P^1$ as above. The preimage
	$\pi^{-1}(P') \subset \Complex P^1$ is a union of two segments (arcs),
	$P_+$ and $P_-$, lying in the upper and the lower half-plane,
	respectively. Extend $P_+$ to connect with $u$; the preimage
	$\wf^{-1}(P_+) \subset \hat N$ consists then of $2n$ paths starting in
	the $2n$ points of $\wf^{-1}(u)$. The comparison of the numbering rule
	for right paths given here (taken from \cite{BenDali}) and the
	definition of the labelling $\hat\nu: \wf^{-1}(u) \to \mathcal A_n$
	given above shows that the following is true:
	
	\begin{proposition}\label{Pp:LabelToCDLabel}
		If $\gamma \subset \wf^{-1}(P_+) \subset \hat N$ is the path
		starting at the point $w \in \wf^{-1}(u)$ with $\hat \nu(w) = x \in
		\mathcal A_n$ then $p(\gamma) \subset N$ is the right path numbered
		$x$ plus a short segment joining it with the point $v \in
		f^{-1}(\pi(u)) \subset N$ lying in the relevant corner.
	\end{proposition}
	
	The map $\LL$ relating to a meromorphic function $f$ the unordered
	collection of its critical values (except $\infty$) is called the {\em
		Lyashko--Looijenga map} (also the LL-map or the branch map).
	
	We can combine the construction of Theorem \ref{th:CDConstel} with the
	correspondence between labelling of ramified coverings and CD
	labelling of the constellations described above. Namely, let $F$ be a
	simple fully real ramified covering, $\hat\nu$, its labelling, $\Const
	= f^{-1}(P)$ is the constellation, and $\mathcal V$ is the CD
	labelling of $\Const$ corresponding to $\hat\nu$.
	
	\begin{corollary}[of Theorem \ref{th:CDConstel} and Proposition \ref{Pp:LabelToCDLabel}]
		For any $Y = (y_0 \DT< y_m)$, $y_0 \DT, y_m \in \Real$ the mapping
		$(F,\hat\nu) \mapsto (\Const,\mathcal V) \bydef {\mathcal
			X}(F,\hat\nu)$ is a bijection between the sets $\LL^{-1}(Y)
		\subset \mathfrak{D}_{m,\lambda}^o$ and
		$\mathfrak{C}_{m,\lambda}^{CD}$.
	\end{corollary}
	
	Consider now pair matchings on $\mathcal{A}_n$; they are identified
	with involutions without fixed points in the permutation group
	$S_{2n}$. Given two such involutions $\delta_1$ and $\delta_2$, let
	$\sigma \bydef \delta_1 \delta_2 \in S_{2n}$. Since $\delta_1 \sigma
	\delta_1 = \delta_2 \delta_1 = \sigma^{-1}$, the cycle decomposition
	of $\sigma$ looks like $c_1 c_1' \dots c_s c_s'$ where $c_k' \bydef
	\delta_1 c_k^{-1} \delta_1$; therefore, the cycles $c_k$ and $c_k'$
	have the same length. Let $\lambda_i$ be the length of $c_i$; denote
	the partition $(\lambda_1 \DT, \lambda_s)$ by
	$\Lambda(\delta_1,\delta_2)$; one has
	$\lmod\Lambda(\delta_1,\delta_2)\rmod = \lambda_1 \DT+ \lambda_s =
	n$. To a pair matching $\delta$ one can relate a graph
	$\Gamma(\delta)$ with the vertex set $\mathcal{A}_n$: two vertices $p$
	and $q$ are joined by an (non-oriented) edge if $\delta(p) = q$. An
	edge union of the graphs $\Gamma(\delta_1)$ and $\Gamma(\delta_2)$ is
	a union of cycles of the lengths $2\lambda_1 \DT, 2\lambda_s$; see
	Fig.~\ref{Fg:Square} for an example.
	
	Let, again, $\wf \in \LL^{-1}(Y) \subset
	\mathfrak{D}_{m,\lambda}^o$. Take its constellation $\Const =
	f^{-1}(P)$ and supply it with a CD labelling --- hence, with a
	numbering of the right paths. Following \cite{BenDali} define a pair
	matching $\delta_k$, $k = 0 \DT, m-2$, as follows: $\delta_k(a) = b$
	if a right path number $a$ and a right path number $b$ go close to one
	another on the two sides of an edge of $\Const$ joining vertices
	colored $k$ and $k+1$. Also define $\delta_{-1} \bydef \tau$
	$\delta_{m-1}(a) \bydef b$ if the right path number $a$ and the right
	path number $b$ share a corner of color $m-1$ (recall that the paths
	number $k$ and $\bar k$ share a corner of color $0$). See an example
	at Fig.~\ref{Fg:Deltai}.

	A particular case of a theorem from \cite{BenDali} is:
	
	\begin{theorem}[\cite{BenDali}]\label{Th:BenDCorres}
		\begin{enumerate}
			\item The pair matchings $\delta_{-1} \DT, \delta_{m-1}$ satisfy the
			conditions
			\begin{align}
				&\Lambda(\delta_k,\delta_{k+1}) = 2^1 1^{n-2} \quad\text{for all
				} k = -1 \DT, m-2 \label{Eq:DeltaKK+1}\\
				&\Lambda(\delta_{-1},\delta_{m-1}) =
				\lambda \label{Eq:Delta0M-1}
			\end{align}
			\item The mapping $\Const \mapsto \Delta(\Const) \bydef (\delta_{-1}
			\DT, \delta_{m-1})$ is a bijection between the set
			$\mathfrak{C}_{m,\lambda}^{CD}$ of simple CD labelled
			$m$-constellations with the profile $\lambda$ and the set
			$\mathfrak{P}_{m,\lambda}$ of sequences of pair matchings
			$\delta_{-1}=\tau, \delta_0 \DT, \delta_{m-1} \in S_{2n}$
			satisfying conditions \eqref{Eq:DeltaKK+1}, \eqref{Eq:Delta0M-1}.
		\end{enumerate}
	\end{theorem}
	Combining the bijection of Theorem \ref{th:CDConstel} with the
	bijection of Theorem \ref{Th:BenDCorres} one obtains
	
	\begin{corollary}[\cite{BenDali}]\label{Cor:BenCorres}
		Let $F \in \LL^{-1}(Y) \subset \mathfrak{D}_{m,\lambda}^o$ and let
		$\Const \bydef f^{-1}(P)$ be the associated constellation and
		$\mathcal V$, the CD labelling corresponding to $\hat\nu$. For any
		$Y$ the mapping $F \mapsto \Delta(f^{-1}(P),\mathcal V)$ is a
		bijection between the sets $\LL^{-1}(Y) \subset
		\mathfrak{D}_{n,\lambda}^o$ and $\mathfrak{P}_{m,\lambda}$.
	\end{corollary}
	
	%\begin{figure}[H]
	%\center
	%\includegraphics[scale=0.6]{RightPath.pdf}
	%\caption{RightPath.pdf}\label{Fg:RightP}
	%\end{figure}

	\begin{figure}[H]
		\center
		\includegraphics*[scale=0.8]{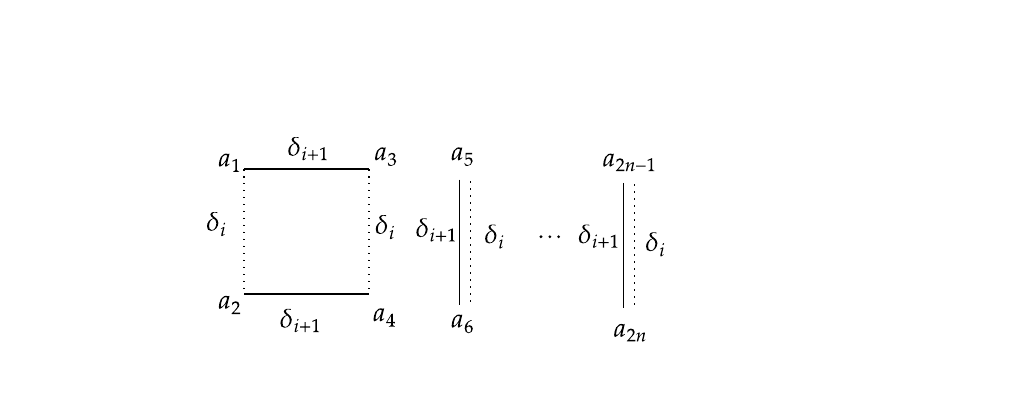}
		\caption{$\Lambda(\delta_{i}, \delta_{i+1})=[2,1^{n-2}]$}\label{Fg:Deltai}
	\end{figure}
	
	\section{A combinatorial correspondence}\label{Sec:Combinat}
	
	Consider a sequence of transpositions $((i_1,j_1)\DT,(i_m,j_m)) \in
	\mathfrak H_{m,\lambda}^{\Real}$; recall that it implies that the
	permutation $x_m \tau x_m \tau \in S_{2n}$ has the cyclic type
	$(\lambda_1, \lambda_1 \DT, \lambda_s, \lambda_s)$ (each part of
	$\lambda$ is repeated twice) where $x_k\bydef (i_1,j_1)\dots(i_k,j_k)$
	for all $k = 1 \DT, m$. Let $\delta_k=(\tau x_{k+1})\tau(\tau
	x_{k+1})^{-1}$ for $k = 0 \DT, m-1$, and $\delta_{-1} \bydef
	\tau$. The permutation $\delta_k$ is conjugate to $\tau$ and therefore
	is an involution without fixed points or, equivalently, a pair
	matching.
	
	\begin{proposition}\label{Pp:TranspoPerfMatch}
		The map $\mathcal P((i_1,j_1) \DT, (i_m,j_m)) \bydef (\tau,\delta_0
		\DT, \delta_{m-1})$ is a $2^m:1$ correspondence between the sets
		$\mathfrak H_{m,\lambda}^{\Real}$ and $\mathfrak{P}_{m,\lambda}$.
	\end{proposition}
	
	\begin{proof}
		Check first that $(\tau,\delta_0 \DT, \delta_{m-1}) \in
		\mathfrak{P}_{m,\lambda}$. Indeed, if $x_{k+1} = x_k(ij)$ then
		$\delta_k\delta_{k+1} = \tau x_k\tau x_k^{-1}\tau \cdot \tau
		x_k(ij)\tau(ij) x_k^{-1}\tau = \tau x_k\tau(ij)\tau(ij) x_k^{-1}\tau =
		\tau x_k (\tau(i)\tau(j)) (ij) (\tau x_k)^{-1} = (cd)(ab)$ where $a
		\bydef \tau x_k(i)$, $b \bydef \tau x_k(j)$, $c \bydef \tau
		x_k\tau(i)$, and $d \bydef \tau x_k\tau(j)$. Once $i \neq j, \tau(i),
		\tau(j)$, the elements $a,b,c,d$ are all distinct. It means that
		$\Lambda(\delta_k,\delta_{k+1}) = [2,1^{n-2}]$ in the notation of
		Section \ref{Sec:Prelim}. The condition $((i_1,j_1)\DT,(i_m,j_m)) \in
		\mathfrak H_{m,\lambda}^{\Real}$ implies $\tau (\tau x_m) \tau (\tau
		x_m)^{-1} = x_m \tau x_{m-1}^{-1} \tau \in B_\lambda^\sim$, hence
		$\Lambda(\tau,\delta_{m-1}) = \lambda$. See Fig.~\ref{Fg:Square} for a
		graphical representation of the permutations $\delta_k$,
		$\delta_{k+1}$.
		
		To check that the correspondence is $2^m:1$ let $\mathcal P((i_1,j_1)
		\DT, (i_m,j_m)) = (\delta_{-1}=\tau,\delta_0 \DT, \delta_{m-1}) \in
		\mathfrak{P}_{m,\lambda}$; fix $k \le m-1$ and denote for convenience
		$i_k \bydef i, j_j \bydef j$, so $x_{k+1} = x_k(ij)$ (where, as above,
		$x_k \bydef (i_1 j_1) \dots (i_k j_k)$). Since $\Lambda(\delta_k,\delta_{k+1}) =
		[2,1^{n-2}]$, one has $\delta_{k-1} \delta_k = (cd)(ab)$ for some $a,
		b, c, d$. By the definition of $\mathcal P$ one has $(ab)(cd) = \tau x_k \tau
		x_k^{-1} \tau \tau x_{k+1} \tau x_{k+1}^{-1} \tau = \tau x_k \tau
		x_k^{-1} x_{k+1} \tau x_{k+1}^{-1} \tau = \tau x_k \tau (ij) \tau
		x_{k+1}^{-1} \tau = \tau x_k (\tau(i) \tau(j)) (ij) x_k^{-1} \tau$,
		whence either $i = (x_k^{-1} \tau)(a)$, $j = (x_k^{-1} \tau)(b)$ or $i = (\tau x_k^{-1}
		\tau)(a)$, $j = (\tau x_k^{-1} \tau)(b)$ (the exchange $i
		\leftrightarrow j$ gives the same solution). Thus, for the fixed $(i_1,
		j_1) \DT, (i_k,j_k)$ there exist exactly two $(ij) =
		(i_{k+1},j_{k+1})$, and the total number of elements in $(\mathcal
		P)^{-1}(\delta_{-1} \DT, \delta_{m-1})$ is $2^m$. 
	\end{proof}
	
	\begin{figure}[H]
		\center
		\includegraphics[height=5cm,keepaspectratio,clip,scale=0.8]{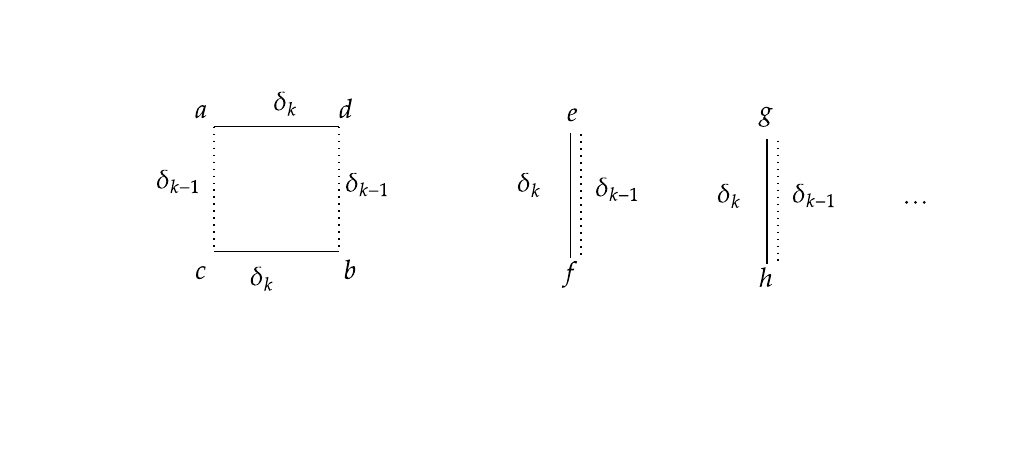} 
		\caption{Perfect Matching}\label{Fg:Square}
	\end{figure}
	
	Now Proposition \ref{Pp:TranspoPerfMatch} and Corollary
	\ref{Cor:BenCorres} imply
	
	\begin{theorem}\label{Th:MainCorresp}
		The composition of the map $\mathcal P$ with the inverse of the
		bijection of Corollary \ref{Cor:BenCorres} is a $2^m:1$
		correspondence between the sets $\mathfrak H_{m,\lambda}^{\Real}$
		and $\LL^{-1}(Y) \subset \mathfrak{D}_{m,\lambda}^o$.
	\end{theorem}
	
	For convenience we recapitulate the list of correspondences that have been done:
	
	\begin{equation*}
		\mathfrak
		H_{m,\lambda}^{\Real}\xleftrightarrow[Prop.\ref{Pp:TranspoPerfMatch}]{{2^m:1}}
		\mathfrak
		P_{m,\lambda}
		\xleftrightarrow[Th.\ref{Th:BenDCorres}]{{1:1}}\mathfrak{C}_{m,\lambda}^{CD}\xleftrightarrow
		[Cor.\ref{Cor:BenCorres}]{{1:1}}\mathfrak{D}_{m,\lambda}^o
	\end{equation*}
	
	\section{An analytico-geometric correspondence}\label{Sec:AlgGeom}
	
	\subsection{Decorated Hurwitz space}
	
	Let, again, $\lambda = (\lambda_1 \DT\ge \lambda_s)$ be a partition of
	$n \bydef \lmod\lambda\rmod \bydef \lambda_1 \DT+ \lambda_s$, and $m >
	0$ be an integer. Denote by $\RHurw_{m,\lambda}$ the set of $4$-tuples
	$(M,\T,f,\nu)$, up to equivalence, such that
	
	\begin{itemize}
		\item $(M,\T,f) \in \RHStand_{m,\lambda}$ (see the definition of
		$\RHStand$ above in the beginning of Section \ref{Sec:Prelim}),
		
		\item $\nu$ is a bijection from the set of critical points of $f$ to
		$\{1 \DT, 2m\}$ such that $\nu(\T(u_i)) = 2m+1-\nu(u_i)$,
		
		\item If $1 \le \nu(a) < \nu(b) \le m$ then $\name{Re}(f(a)) <
		\name{Re}(f(b))$.
	\end{itemize}
	
	There is a simple but important technical result:
	
	\begin{proposition}\label{Pp:NoAutom}
		Let $n+s > 4$ where $n \bydef \lmod\lambda\rmod$ and $s \bydef
		\#\lambda$ and let $(M,f,\nu) \in \Hurw_{m,\lambda}$. The the triple
		$(M,f,\nu)$ has no automorphisms: if $Q: M \to M$ is a holomorphic
		diffeomorphism such that $f \circ Q = f$ and $\nu \circ Q = \nu$
		then $Q = \name{id}$.
	\end{proposition}
	
	To prove it, we need the following
	
	\begin{lemma}\label{Lm:FixInvol}
		A holomorphic involution of a compact complex curve of genus $g$ has
		at most $2g+2$ fixed points. 
	\end{lemma}
	
	\begin{proof}
		Let $M$ be the curve in question, and $Q: M \to M$, the
		involution; denote $N \bydef M/Q$.
		
		Prove that $N$ is a smooth complex curve. Indeed, if $a \in M$ is
		not a fixed point then $N$ is smooth in a neighbourhood of $Q(a)$ by
		the implicit function theorem. If $a$ is a fixed point, then $Q(z) =
		z^2$ for some coordinate $z$ in a neighbourhood of $a$. Since
		$\Complex/(z \mapsto z^2)$ is biholomorphic to $\Complex$, $N$ is
		smooth in a neighbourhood of $Q(a)$, too.
		
		Let $h$ be the genus of the curve $N$ and $k$, the number of fixed
		points of $Q$. By the Riemann--Hurwitz formula, $2-2g-k =
		2(2-2h-k)$, so $k = 2g+2-4h \le 2g+2$.
	\end{proof}
	
	\begin{proof}[Proof of Proposition \ref{Pp:NoAutom}]
		It follows from $\nu \circ Q = \nu$ and the bijectivity of $\nu$
		that $Q$ maps every finite critical point of $f$ to itself. Such a
		critical point $a$ is simple, so there exists a coordinate $z$ on
		$M$ near $a$ such that $f(z) = z^2$. Hence, $Q(z) = \pm z$ in this
		coordinate, which implies that $Q$ is an involution.
		
		The function $f$ has $m$ simple critical points and $s$ poles (of
		multiplicities $\lambda_1, \DT, \lambda_s$); the degree of $f$ is
		$n$. If $g$ is a genus of $M$ then it follows from the
		Riemann--Hurwitz formula that $m = 2g+n+s-2$. Therefore $f$ has
		$2g+n+s-2$ simple critical points, which are all fixed points of
		$Q$. By Lemma \ref{Lm:FixInvol} one has $2g+n+s-2 \le 2g+2$, so $n+s
		\le 4$, contrary to the assumption.
	\end{proof}
	
	\begin{Remark}
		If $n+s \le 4$, the triple $(M,f,\nu)$ may have an automorphism. For
		example, let $M$ be a hyperelliptic curve of any genus $g$ (or any
		curve of genus $g = 0$ or $g = 1$), and let $Q: M \to M$ be a
		hyperelliptic involution. $Q$ has $2g+2$ fixed points, and $M/Q =
		\Complex P^1$. Take for $f: M \to M/Q = \Complex P^1$ the standard
		projection; here $n = s = 2$ and $Q$ is the required automorphism
		(independent on $\nu$ because every critical point of $f$ is mapped
		to itself).
	\end{Remark}
	
	\begin{corollary}\label{Cr:HurwSmooth}
		If $n+s > 4$ then the spaces $\HStand_{m,\lambda}$ and
		$\Hurw_{m,\lambda}$ are smooth.
	\end{corollary}
	
	Smoothness follows from the absence of automorphisms; the proof is the
	same as for the standard Hurwitz spaces (see e.g.\ \cite{RomWew}).
	
	\begin{corollary}\label{Cr:TUnique}
		If $(M,\T,f,\nu) \in \RHurw_{m,\lambda}$ and $n+s > 4$ then $M$ and
		$f$ determine the real structure $\T$ uniquely.
	\end{corollary}
	
	Indeed, if $\T_1$ and $\T_2$ are two real structures then $\T_1 \circ
	\T_2 = \T_1 \circ \T_2^{-1}$ is an automorphism of the triple
	$(M,f,\nu)$.
	
	Thus, by a slight abuse of notation one may write $\RHurw_{m,\lambda}
	\subset \Hurw_{2m,\lambda+\lambda}$ where by $\lambda + \lambda$ we
	denote the partition $(\lambda_1, \lambda_1, \DT, \lambda_s,
	\lambda_s)$ (every part of $\lambda$ is repeated twice). 
	
	If the genus $g$ of the curve $M$ is equal to $0$ or $1$, consider the
	Lie group $\name{Aut}^0(M)$ of holomorphic automorphisms of $M$
	homotopic to identity; one has $\dim \name{Aut}^0(M) = 3$ for $g=0$
	and $\dim \name{Aut}^0(M) = 1$ for $g=1$. The group $\name{Aut}^0(M)$
	acts on the space $\Hurw_{m,\lambda}$, and it is convenient to
	consider auxiliary smooth submanifolds $\widetilde{\Hurw_{m,\lambda}}
	\subset \Hurw_{m,\lambda}$ transversal to the orbits of the action. To
	define them, take a point $0 \in M$ for $g=1$ and points $0,1,\infty
	\in M$ for $g = 0$, and say that $(M,f,\nu) \in
	\widetilde{\Hurw_{m,\lambda}}$ if $0$ (respectively, $0,1,\infty$) is
	a critical point of $f$ numbered $1$ (are critical points numbered $1,
	2, 3$, respectively). The smoothness of
	$\widetilde{\Hurw_{m,\lambda}}$ is proved in the same way as the
	smoothness of $\Hurw_{m,\lambda}$ in Corollary \ref{Cr:HurwSmooth}.
	
	For genera $g \ge 2$ the group $\name{Aut}^0(M)$ is trivial, and we
	take $\widetilde{\Hurw_{m,\lambda}} = \Hurw_{m,\lambda}$ for
	convenience. We also denote $\widetilde{\HStand_{m,\lambda}} \bydef
	\Phi_1(\widetilde{\Hurw_{m,\lambda}})$. Also take by definition
	$\widetilde{\RHurw_{m,\lambda}} \bydef \RHurw_{m,\lambda} \cap
	\widetilde{\Hurw_{2m,\lambda+\lambda}}$.
	
	Let's describe the tangent space to $\Hurw_{m,\lambda}$ at a point
	$(M,f,\nu)$.
	
	\begin{proposition}\label{Pp:THurw}
		Let $f: M \to \Complex P^1$ have poles $a_1, \DT, a_s$ where the
		multiplicity of $a_i$ is $\lambda_i$. Then the tangent space to
		$\Hurw_{m,\lambda}$ at the point $(M,f,\nu)$ is naturally identified
		with the set of meromorphic functions on $M$ having the same poles
		as $f$ and such that the multiplicity of the pole $a_i$ does not
		exceed $\lambda_i+1$.
	\end{proposition}
	
	\begin{proof}
		Let $M$ be a complex curve, and $p: N \to M$, its universal cover
		and $\Gamma$, the translation group, so that $M = N/\Gamma$ and $p
		\circ g = p$ for any $g \in \Gamma$. Actually, either $N = M =
		\Complex P^1$ and $\Gamma$ is trivial, or $N = \Complex$ and $\Gamma
		= \Integer^2$ is a lattice, or $N = \Upp$ and $\Gamma \subset
		\name{PSL}_2(\Integer)$ is a discrete (Fuchsian) subgroup for the
		curves $M$ of genus $0$, $1$, or $\ge 2$, respectively. Denote by $F
		\bydef f \circ p: N \to \Complex P^1$; thus $F \circ Q = F$ for any
		$Q \in \Gamma$.
		
		Thus, $\Hurw_{m,\lambda}$ can be identified with the space of pairs
		$(F,\Gamma)$ where $F: N \to \Complex P^1$ is a $\Gamma$-invariant
		meromorphic function. Let $\gamma(t) = (F_t,\Gamma_t)$ be a curve on
		$\Hurw_{m,\lambda}$ with $\gamma(0) = (F,\Gamma)$; such curves, up
		to tangency at $0$, form the tangent space we are looking for. The
		function $F_t$ is $\Gamma_t$-invariant; in particular, $\Gamma_t$
		sends poles of $F_t$ into poles. Since $\Gamma_t$ is discrete, it is
		fully determined, for small $t$, by $F_t$ (as a group acting on
		$N$); so, $F_t$ determines the curve $\gamma(t)$ completely.
		
		Take a pole $b_i$ of $F$ such that $p(b_i) = a_i$, and let $z$ be a
		coordinate in a neighbourhood of $b_i$.  Then $F_t(z) =
		\frac{P_t(z)}{(z-b(t))^{\lambda_i}}$ where $P_t(z)$ is a polynomial
		(in $z$) of degree $\le \lambda_i-1$ and $b(0) = b_i$. Then the
		curve $\gamma$ up to tangency at $0$ is determined by the value of
		the derivative $\left.\frac{dF_t}{dt}(z)\right|_{t=0} =
		\frac{\left.\frac{dP_t(z)}{dt}\right|_{t=0}\cdot (z-b_i) - \lambda_i
			P_0(z)}{(z-b_i)^{\lambda_i+1}} \bydef
		\frac{Q(z)}{(z-b_i)^{\lambda_i+1}}$ for some polynomial $Q$ of
		degree $\le \lambda_i$. Obviously,
		$\left.\frac{dF_t}{dt}(z)\right|_{t=0}$ is $\Gamma$-invariant, so
		there exists a function $\phi: M \to \Complex P^1$ with
		$\left.\frac{dF_t}{dt}(z)\right|_{t=0} = \phi \circ p$. The
		meromorphic function $\phi$ has poles at $a_i$ only, of the
		multiplicities $\le (\lambda_i+1)$.
	\end{proof}
	
	\subsection{Lyashko--Looijenga map and real Hurwitz space}

	Let $(M,f,\nu) \in \Hurw_{m,\lambda}$; denote $\OLL(f) = (f(u_1) \DT,
	f(u_m)) \in \Complex^m$ where $u_i$ is the finite critical point of
	$f$ such that $\nu(u_i) = i$. If $\Complex^{(m)}$ is a $m$-th
	symmetric power of $\Complex$ (set of unordered collections, possibly
	with repetitions, of $m$ complex numbers) and $\Phi_2: \Complex^m \to
	\Complex^{(m)}$ is the quotient map (forgetting the ordering), then
	the map $\LL: \HStand_{m,\lambda} \to \Complex^{(m)}$ such that $\LL
	\circ \Phi_1 = \Phi_2 \circ \OLL$ is called {\em Lyashko--Looijenga map}
	(also the LL-map or the branch map). We will call $\OLL$ {\em an
		ordered LL-map}.
	
	Apparently, the map $\OLL$ is invariant with respect to the action of
	$\name{Aut}^0(M)$ mentioned above. A direct calculation (see
	e.g.~\cite{ELSV}) shows that the dimension of the set of
	$\name{Aut}^0(M)$-orbits is equal to the number of critical points of
	the function $f$ for all genera. So both $\LL$ restricted to
	$\widetilde{\HStand_{m,\lambda}}$ and $\OLL$ restricted to
	$\widetilde{\Hurw_{m,\lambda}}$ are maps between spaces of equal
	dimension. (Recall that the tilde is to be ignored if $g \ge 2$.)
	
	\begin{theorem}\label{Th:LocalInj}
		The map $\OLL: \widetilde{\Hurw_{m,\lambda}} \to \Complex P^m$ is a
		local injection (in a neighbourhood of any point).
	\end{theorem}
	
	\begin{proof}
		Using Corollary \ref{Cr:HurwSmooth} and the equidimensionality
		mentioned above we can derive the statement from the inverse
		function theorem; one needs to prove that the derivative $\OLL'(f)$
		(where $\OLL$ is restricted to $\widetilde{\Hurw_{m,\lambda}}$) has
		trivial kernel. In the calculations below, as usual, $D(A)$ for a
		divisor $A$ on $M$ means the set or meromorphic functions $\psi$ on
		$M$ such that the divisor $(\psi)+A \ge 0$
		
		Let first $g \ge 2$. By Proposition \ref{Pp:THurw} one is to
		consider $\OLL(f+th) - \OLL(f)$ up to $o(t)$, where $h \in
		D((\lambda_1+1)(a_1 + \T(a_1)) \DT+ (\lambda_s+1) (a_s + \T(a_s)))$.
		Let $u$ be a non-degenerate critical point of $f$: $f'(u) = 0$,
		$f''(u) \ne 0$. The critical point of $f+th$ is $u+tw +o(t)$ such
		that $(f + th)'(u+tw +o(t))=0$. It means $t (f''(u) w + h'(a)) = 0$,
		so $w = -h'(u)/f''(u)$. Now the component of the $\OLL'(f)(h)$
		corresponding to $u$ is $\frac1t ((f+th)(u+tw+o(t))-f(u)) + o(t) =
		h(u)$; so $\OLL'(f)(h) = (h(u_1), h(\T(u_1)) \DT, h(u_m),
		h(\T(u_m)))$.
		
		Thus if the tangent vector $h$ belongs to the kernel of $\OLL'(f)$,
		it should take zero value at all finite critical points of the
		function $f$. In other words, $h \in D(A)$ where $A = -(u_1 +
		\T(u_1) \DT- (u_m + \T(u_m)) + (\lambda_1+1)(a_1 + \T(a_1)) \DT+
		(\lambda_s+1) (a_s + \T(a_s)))$. Then $\deg A = -2m + 2\sum_{i=1}^s
		(\lambda_i+1) = -2m + 2n + 2s$. Since $2m = 2n+2g+2s-2$ (see the
		proof of Proposition \ref{Pp:NoAutom}) where $g$ is the genus of
		$M$, one has $\deg A = 2-2g < 0$, and $h \equiv 0$. 
		
		If the genus $g = 1$ then the subset
		$\widetilde{\Hurw_{2m,\lambda+\lambda}}$ is defined by the equation
		$f'(0) = 0$. So if the vector $h$ is tangent to
		$\widetilde{\Hurw_{2m,\lambda+\lambda}}$ then $(f+th)'(0) = 0$, so
		that $h'(0) = 0$. Then one performs the same calculations as above
		but the divisor $A$ contains the point $0$ with the coefficient
		$-2$; this gives $\deg A = -2 < 0$. A similar calculation for $g =
		0$ gives $\deg A = -6 < 0$.
		
		Thus $\OLL$ is a local injection on
		$\widetilde{\Hurw_{2n,\lambda+\lambda}}$ and hence on its subset
		$\widetilde{\RHurw_{n,\lambda}}$.
	\end{proof}
	
	\begin{corollary}[of Theorem \ref{Th:LocalInj} and the remark after
		Corollary \ref{Cr:TUnique}] The restriction of the map $\OLL$ to
		$\widetilde{\RHurw_{m,\lambda}}$ is a local injection (in a
		neighbourhood of any point).
	\end{corollary}

	\begin{theorem}\label{Th:Local2m}
		The local multiplicity of the Lyashko--Looijenga map $\LL$ near the
		point $F \in \mathfrak{RH}_{m,\lambda}$ is $2^m$.
	\end{theorem}
	
	\begin{proof}
		The involution $\T$ has no fixed points. It follows from last
		condition in the definition of $\RHurw_{m,\lambda}$ that the
		forgetful map $\Phi_1: \RHurw_{m,\lambda} \to \RHStand_{m,\lambda}$
		is a $2^m$-sheeted (unramified) covering. (Formally, it is a
		covering with a fiber $\{-1,1\}^m$. To define the trivialization map
		$\xi$ consider a point $F = (M,\T,f,\nu)$ such that $\Phi_1(F) =
		(M,\T,f)$ is in a neighbourhood of of $\Phi_1(F_0) = (M_0,\T_0,f_0)$
		where $F_0 = (M_0,\T_0,f_0,\nu_0)$. It follows from the last
		condition in the definition of $\RHurw_{m,\lambda}$ that for every
		$i = 1 \DT, m$ the critical point $u$ of $f$ such that $\nu(u) = i$
		is close to either the critical point $u_0$ of $f_0$ such that
		$\nu_0(u_0) = i$ or to $\T_0(u_0)$. Define $\eps_1 = 1$ in the first
		case and $\eps_i = -1$ in the second; then $\xi(F) = (\eps_1 \DT,
		\eps_m)$ is the required trivialization.)
		
		Therefore $\Phi_1$ is a local diffeomorphism at any point. The map
		$\OLL$ is a local diffeomorphism, too, by Theorem
		\ref{Th:LocalInj}. Thus, the local multiplicity of the map $\LL$
		is the same as that of the map $\LL \circ \Phi_1 = \Phi_2 \circ
		\OLL$, and the same as that of the map $\Phi_2$.
		
		For $F \in \mathfrak{RH}_{m,\lambda}$ the image $\LL(\Phi_1(F)) =
		\Phi_2(\OLL(F))$ is a multi-set of real numbers $y_1 \DT< y_m$, each
		repeated twice. Take small generic $\delta_1 \DT, \delta_m > 0$ and
		let $z = (z_1 \DT, z_{2m})$ be such that $\Phi_2(z) = y_\delta
		\bydef \{y_1 + i\delta_1, y_1 - i\delta_1 \DT, y_m + i\delta_m, y_m
		- i\delta_m\}$. Then the coordinates $z_1$ and $z_{2m}$ are $y_1 +
		i\delta_1$ and $y_1 - i\delta_1$ in either order, the coordinates
		$z_2$ and $z_{2m-1}$ are $y_2 + i\delta_2$ and $y_2 - i\delta_2$,
		etc. Hence, $\Phi_2^{-1}(y_\delta)$ contains $2^m$ points, and the
		theorem is proved.
	\end{proof}
	
	\section{From geometry to combinatorics and vice versa}\label{Sec:GeomCombin}
	Let, as above, $F \bydef (M,\T,f) \in \mathfrak{RH}_{m,\lambda}$ and
	let $y_0 \DT< y_{m-1} \in \Real$ be the critical values of $f$; denote
	$Y \bydef [y_0, y_0, \DT, y_{m-1}, y_{m-1}] \in \Real^{(2m)}$. Take a
	regular value $u \in \Real$ of $f$, $u < y_0$, and let $\nu: f^{-1}(u)
	\to \mathcal A_n$ be a bijection. Then $(F,\nu) = (M,\T,f,\nu) \in
	\mathfrak{D}_{m,\lambda}^o$ (the set of labelled simple fully real
	ramified coverings) and $\LL(F) = Y$.
	
	Consider loops $\alpha_0 \DT, \alpha_{m-1}$ shown
	Fig.~\ref{Fg:Monodromy}. Denote by $u_k \in S_{2n}$ ($k = 0 \DT,
	m-1$)the monodromy of $f$ around the loop $\alpha_k$, and let
	$\delta_k \bydef \tau u_k$.
	
	%\begin{figure}
	%\includegraphics[scale=0.6]{Labelat0.pdf}
	%\caption{proof for the labelling}\label{Fg:Labelat0}
	%\end{figure}
	
	\begin{theorem}\label{Th:FuncToDelta}
		The correspondence $(F,\nu) \mapsto (\delta_{-1} \bydef \tau,
		\delta_0 \DT, \delta_{m-1})$ is a one-to-one map from $\LL^{-1}(Y)
		\subset \mathfrak{D}_{m,\lambda}^o$ to $\mathfrak
		H_{m,\lambda}^{\Real}$.
	\end{theorem}
	
	To prove the theorem take first some $\eps_0 \DT, \eps_{m-1} > 0$; they
	are assumed to be small enough, but otherwise the construction below
	does not depend on them. Choose $\widetilde{F} \bydef
	(\widetilde{M}, \widetilde{\T}, \widetilde{f}, \widetilde{\nu}) \in
	\RHurw_{m,\lambda}$ close to $F$ and such that
	$\LL(\widetilde{F}) = (z_0, \bar z_0 \DT, z_{m-1}, \bar z_{m-1})$ where
	$z_k, \bar z_k = y_k \pm i\eps_k$, $k = 0 \DT, m-1$.
	
	Consider loops $\gamma_0 \DT,\gamma_{m-1}, \overline{\gamma_{m-1}^{-1}} \DT,
	\overline{\gamma_0^{-1}}$ like in Fig.~\ref{Fg:Monodromy}; they
	start and finish at $u$. The critical values $z_k$ and $\bar z_k$ are
	simple, so the monodromy of $\widetilde{f}$ along the loop
	$\gamma_k$ is a transposition $(i_k,j_k) \in S_{2n}$.
	
	\begin{proposition}\label{Pp:MndrPerturb}
		Relating to a point $F \in \mathfrak{RH}_{m,\lambda}$ the sequence
		of transpositions $((i_0,j_0) \DT, (i_{m-1},j_{m-1}))$ defines a
		$2^m$-valued correspondence $\mathfrak{RH}_{m,\lambda} \to \mathfrak
		H_{m,\lambda}^{\Real}$.
	\end{proposition}
	
	\begin{proof}
		The monodromy of $\widetilde{f}$ along the loop $\overline{\gamma_k}$
		is $(\tau(i_k),\tau(j_k)) = \tau (i_k j_k) \tau$ where $\tau =
		(1,n+1)\dots(n,2n)$ because the function $\widetilde{f}$ is
		real. The loop $\gamma_0 \DT\gamma_{m-1}
		\overline{\gamma_{m-1}^{-1}}\DT \overline{\gamma_0^{-1}}$ is
		homotopic to a loop circling around the point $\infty$, so the total
		monodromy $(i_0,j_0)\dots(i_{m-1},j_{m-1})
		(\tau(i_{m-1}),\tau(j_{m-1})) \dots (\tau(i_0),\tau(j_0))$ belongs
		to $B_\lambda^\sim$. This proves that the image of the
		correspondence lies in $\mathfrak H_{m,\lambda}^{\Real}$.
		
		By Theorem \ref{Th:Local2m} there are $2^m$ ways to choose an
		approximation $\widetilde{\phi}$. The critical values $z_k, \bar
		z_k$ of $\widetilde{f}$ are simple, and the loops $\gamma_0
		\DT,\gamma_{m-1}, \overline{\gamma_{m-1}^{-1}} \DT,
		\overline{\gamma_0^{-1}}$ cut $\Complex P^1$ into disks, each one
		containing at most one critical value. By the classical theorem (see
		e.g.\ \cite{LandoZvonkin}) the critical values $z_k, \bar z_k$ and
		the monodromies around the loops determine $\widetilde{f}$
		uniquely. So the correspondence is $2^m$-valued.
	\end{proof}
	
	\begin{figure}[H]
		\center
		\includegraphics[scale=0.8]{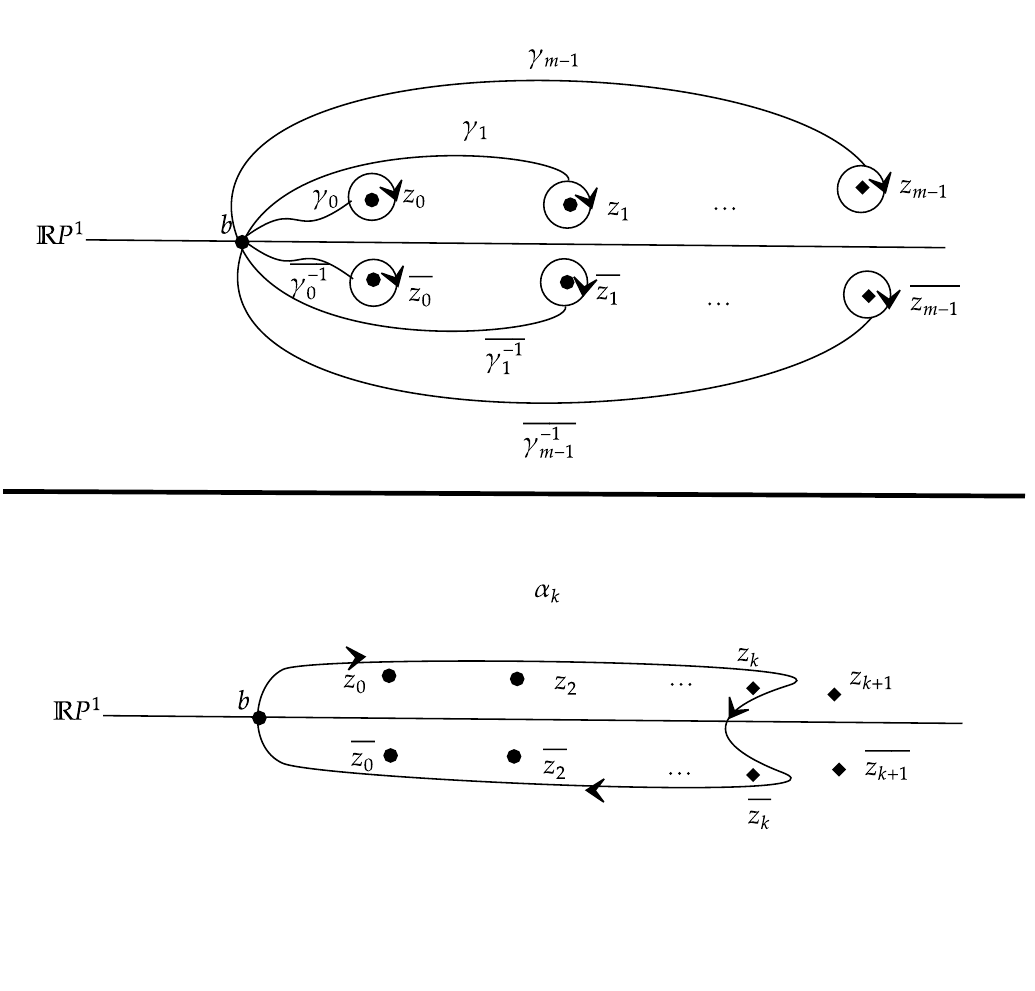} 
		\caption{Monodromy}\label{Fg:Monodromy}
	\end{figure}
	
	\begin{proof}[Proof of Theorem \ref{Th:FuncToDelta}]
		Compose the correspondence of Proposition \ref{Pp:MndrPerturb} with
		the (combinatorial) map of Theorem \ref{Th:MainCorresp} to obtain a
		sequence $(\delta_{-1} = \tau, \delta_0 \DT, \delta_m)$. Clearly,
		$\tau \delta_k = x_k \tau x_k^{-1} \tau$ is a monodromy of
		$\widetilde{f}$ around the loop $\gamma_0 \dots \gamma_k
		\overline{\gamma_k^{-1}} \dots \overline{\gamma_0^{-1}}$, which is
		homotopic to $\alpha_k$.
		
		Now let $\eps_0 \DT, \eps_k \to 0$. Then the critical values $z_k,
		\bar z_k \to y_k$ and $\widetilde{f} \to f$. The monodromy of
		$\widetilde{f}$ around $\alpha_k$ does not change because the
		critical values do not intersect the loop. Therefore, the monodromy
		of $f$ itself around $\alpha_k$ is the same. This implies that
		$\delta_k$, $k = 0 \DT, m$, do not depend on the choice of the
		approximation $\widetilde{\phi}$. Thus, the composition is a
		one-valued correspondence, i.e., a map. Since the correspondence of
		Proposition \ref{Pp:MndrPerturb} is $2^m$-valued, and the map of
		Theorem \ref{Th:MainCorresp} is $2^m : 1$, the composition map is
		one-to-one.
	\end{proof}
	
	Now we have two bijective maps between the spaces $\LL^{-1}(Y) \subset
	{\mathfrak D}_{m,\lambda}^o$ and $\mathfrak H_{m,\lambda}^{\Real}$:
	the map of Theorem \ref{Th:FuncToDelta} and that of Theorem
	\ref{Th:MainCorresp}.
	
	\begin{theorem}
		The correspondences between $\mathfrak{D}^o_{m,\lambda}$ and
		$\mathfrak{H}_{m,\lambda}^\Real$ obtained in Theorem \ref{Th:FuncToDelta}
		and Theorem \ref{Th:MainCorresp} are mutually inverse.
	\end{theorem}
	
	\begin{proof}
		Let $\delta = (\tau = \delta_{-1}, \delta_0 \DT, \delta_m) \in
		\mathfrak{H}_{m,\lambda}^\Real$ and ${\mathfrak I}(\delta) \bydef \delta'
		\bydef (\tau, \delta'_0 \DT, \delta'_m)$ where $\mathfrak I:
		\mathfrak{H}_{m,\lambda}^\Real \to \mathfrak{H}_{m,\lambda}^\Real$ is the
		composition of the two maps. Prove that $\delta_k' = \delta_k$ using
		induction on $k$; the base $k = -1$ is trivial.
		
		Now let $\delta_i' = \delta_i$ for $i = -1 \DT, k-1$. Denote by $f:
		M \to \Complex$ the image of $\delta$ under the map of Theorem
		\ref{Th:MainCorresp}; $\LL(f) = [y_0, y_0 \DT, y_{m-1},
		y_{m-1}]$. Fix a perturbation $\check{f}$ of $f$ with
		$\LL(\check{f}) = [y_0 + i\eps_0, y_0 - i\eps_0 \DT, y_{m-1} +
		i\eps_{m-1}, y_{m-1} - i\eps_{m-1}]$. If $\sigma_0 \DT,
		\sigma_{m-1}$ are monodromies of $\check{f}$ around the loops
		$\gamma_0 \DT, \gamma_{m-1}$ then $(\sigma_0,\dots,\sigma_{m-1})$
		belongs to $B_\lambda^{\sim}$.
		
		Prove that $\delta'_k = \delta_k$. Suppose that
		$\delta_{k-1}\delta_k = (ab)(cd)$ (cf.\ computation in Proposition
		\ref{Pp:TranspoPerfMatch}). Once $\delta_{k-1}\delta_k =
		\delta_{k-1}^{-1}\delta_k = \delta_{k-1}^{-1}\tau\tau\delta_k =
		(\tau\delta_{k-1})^{-1}\tau\delta_k$, the permutation $(ab)(cd)$ is
		the monodromy of $\check{f}$ around the loop $\beta_k = \alpha_{k-1}
		\alpha_k$ (see Figure \ref{Fg:Monodromy}). The loop $\beta_k$ is
		homotopic to the loop shown in Figure \ref{Fg:Loop}. This loop does
		not intersect the segments joining $y_s + i\eps_s$ with $y_s -
		i\eps_s$, $s = 0 \DT, k$, so $(ab)(cd)$ is also the monodromy of $f$
		around the loop $\beta_k$, and does not depend on the choice of the
		perturbation $\check{f}$.
		
		The right paths labelled $a,b,c,d$ near the critical point colored
		$k$ are shown in Fig.~\ref{Fg:ConstMonod}: the dotted line is the
		preimage of the segment between $y_{k-1}$ and $y_k$ going slightly
		above the real axis, the dashed lines are the right paths (their
		labels circled), and the boldface line is the preimage of the loop
		$\beta_k$. Indeed the preimage of $\beta_k$ first follows the right
		path path $a$, then goes around $y_k$. Then $\beta_k$ crosses the
		real axis twice, so its preimage crosses the dotted line staying
		between the right paths, then crosses it again and goes back to the
		critical point colored $k-1$.
		
		The loop $\beta_k$ for most of its length goes very close to the
		real axis --- so, its preimage stays close to the right paths. This
		means that the right path $a$ goes close to the right path $b$ near
		$y_k$, and so $a$ and $b$ correspond to one another under the pair
		matching of Theorem \ref{Th:BenDCorres}. A similar reasoning shows
		the same for $c$ and $d$. Thus $\delta_{k-1} \delta_k' = (ab)(cd) =
		\delta_{k-1} \delta_k$, hence $\delta_k' = \delta_k$ finishing the
		proof by induction. 
	\end{proof}
	
	\begin{figure}
		\center
		\includegraphics[scale=0.6]{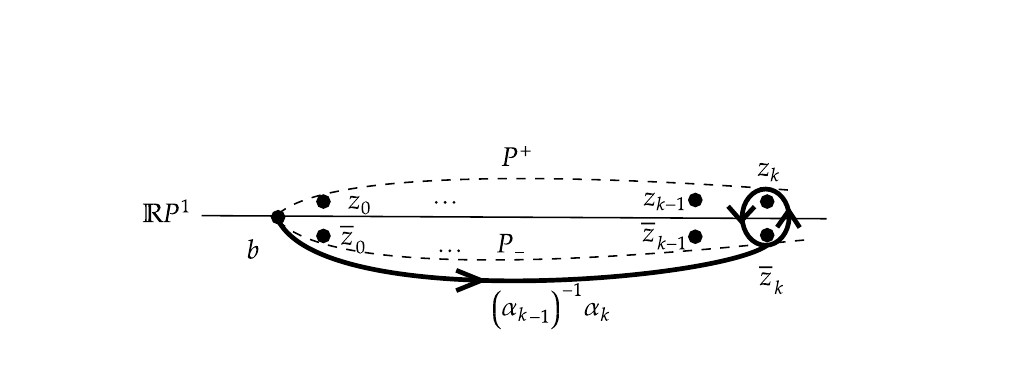} 
		\caption{The loop $\delta_{k-1}\delta_k$}\label{Fg:Loop}
	\end{figure}
	
	\begin{figure}[H]
		\center
		\includegraphics[scale=0.6]{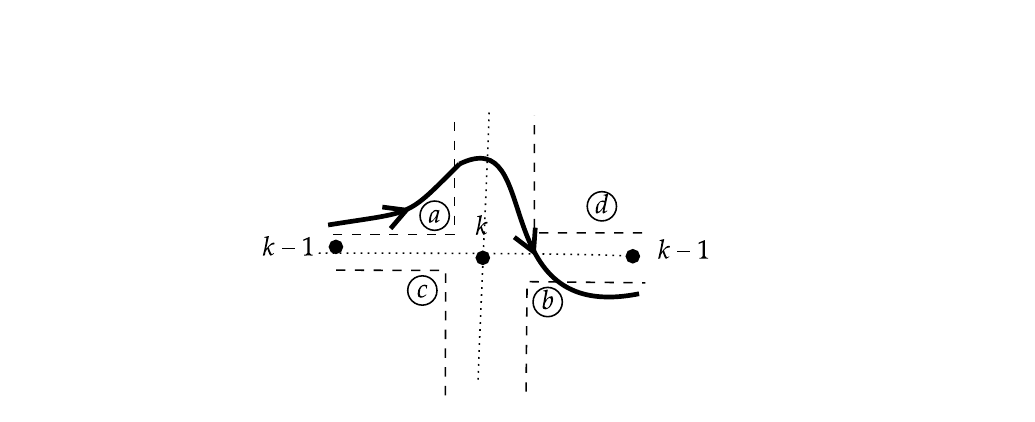} 
		\caption{Right paths and loop}\label{Fg:ConstMonod}
	\end{figure}

\end{document}